\documentclass{amsart}

\usepackage{algorithm}
\usepackage[colorlinks,naturalnames=false,pdfborder={0 0 0},hypertexnames=false,breaklinks]{hyperref}
\usepackage[capitalize,nameinlink]{cleveref}
\usepackage{amsfonts}
\usepackage{amsmath}
\usepackage{color}
\usepackage{algorithmic}
\usepackage[algo2e,linesnumbered,lined,ruled,commentsnumbered]{algorithm2e}
\usepackage{amsbsy}
\usepackage{subcaption}
\Crefname{ALC@unique}{Line}{Lines} 

\usepackage{xargs}
\usepackage[colorinlistoftodos,prependcaption,textsize=tiny,textwidth=2.2cm,shadow]{todonotes}
\newcommandx{\ccom}[2][1=]{\todo[linecolor=red,backgroundcolor=red!25,bordercolor=red,#1]{crp: #2}}
\newcommandx{\ecom}[2][1=]{\todo[linecolor=blue,backgroundcolor=blue!25,bordercolor=blue,#1]{em: #2}}

\newtheorem{proposition}{Proposition}[section]
\newtheorem{lemma}[proposition]{Lemma}
\newtheorem{theorem}[proposition]{Theorem}
\newtheorem{corollary}[proposition]{Corollary}

\begin{document}
\title[Generalized Adaptive Partition-based Method for TSSP]{Generalized Adaptive Partition-based Method for Two-Stage Stochastic Programs with Fixed Recourse}
\author{Cristian Ramirez-Pico}
\author{Eduardo Moreno}
\address{Faculty of Engineering and Sciences, Universidad Adolfo Ib\'a\~nez, Santiago, Chile}
\email{cristian.ramirez@edu.uai.cl,eduardo.moreno@uai.cl} 
\thanks{Supported by CONICYT-Fondecyt Regular 1161064}

	\begin{abstract}
		We present a method to solve two-stage stochastic problems with fixed recourse when the uncertainty space can have either discrete or continuous distributions. Given a partition of the uncertainty space, the method is addressed to solve a discrete problem with one scenario for each element of the partition (sub-regions of the uncertainty space). Fixing first stage variables, we formulate a second stage subproblem for each element, and exploiting information from the dual of these problems, we provide conditions that the partition must satisfy to obtain the optimal solution. These conditions provide guidance on how to refine the partition, converging iteratively to the optimal solution. Results from computational experiments show how the method automatically refines the partition of the uncertainty space in the regions of interest for the problem. Our algorithm is a generalization of the adaptive partition-based method presented by Song \& Luedtke for discrete distributions, extending its applicability to more general cases. 
	\end{abstract}

\maketitle
	
	\section{Introduction}
	We study the following two-stage stochastic program (TSSP) with fixed recourse 
	\begin{equation}
	\min \left\lbrace c^\top x + \mathbb{E}\left[ \mathcal{Q}(x,\xi)\right] ~|~ x\in \mathcal{X} \right\rbrace
	\label{TSSP}
	\end{equation}
	where $\mathcal{X} \subseteq \mathbb{R}^{n}$ is a set assumed to be non-empty closed, $\xi$ is a random vector in the probability space $\left(\Omega,\mathcal{A},\mathbb{P}\right)$ containing the random elements $\big\lbrace h^\xi, T^\xi \big\rbrace$, and second-stage subproblem
	\begin{equation}
	\mathcal{Q}(x,\xi):=\min \left\lbrace q^\top y  ~|~ W y=h^\xi  - T^\xi x, ~y\ge0 \right\rbrace,
	\label{secondst}
	\end{equation}
	where fixed recourse matrix $W\in \mathbb{R}^{m\times n}$, deterministic costs $q \in \mathbb{R}^n$, random technology matrix $T^\xi \in \mathbb{R}^{m\times n}$ and random right-hand side (RHS) vector $h^\xi \in \mathbb{R}^m$. Furthermore, we assume that there exists $\bar{x}$ such that $\mathcal{Q}(\bar{x},\xi)$ is feasible and bounded in the whole outcome space $ \Omega$.
	Note that the support of the uncertainty set $\xi$ can be either continuous or discrete.

	In this paper, we propose a method to solve TSSPs  by iteratively and automatically aggregating the uncertainty set into a small number of scenarios and dissagregating them based on the information of dual subproblem variables. This approach yields a smaller version of the original stochastic problem by reducing both the number of variables and the number of constraints by an equivalent deterministic formulation of~\cref{TSSP}.  For the case of discrete distributions, this idea has been called the \emph{adaptive partition-based method} (APM) by Song \& Luedtke~\cite{song2015adaptive}, and it is based on the results of Espinoza \& Moreno~\cite{espinoza2014primal} and Bienstock \& Zuckerberg~\cite{bienstock2010solving}.  We present an alternative and more general proof that allows us to extend APM to a more general setting, in particular, to deal TSSPs  with continuous distributions for $\Omega$.

Let $P\subseteq \Omega$, and let  $T^P = \mathbb{E}\big[ T^\xi|P\big]$ and $h^P = \mathbb{E}\big[ h^\xi|P\big]$ be the conditional expectations of the components of $\xi$ given $P$. We denote the \emph{aggregated subproblem} as
\begin{equation}
\mathcal{Q}(x,\mathbb{E}\left[ \xi |P \right])=\min \left\lbrace q^\top y  ~|~ W y=h^P  - T^P x, ~y\ge0 \right\rbrace
\label{secondst_agg}
\end{equation}

The contribution of this paper is to provide conditions for a partition $\mathcal{P}$ of $\Omega$ such that the solution of Problem \cref{TSSP} is equivalent to solving
	\begin{equation}
	\min_{x\in\mathcal{X}} \left\{ c^\top x +  \sum_{P\in\mathcal{P}}\mathcal{Q}\left(x,\mathbb{E}\left[\xi|P\right]\right) \cdot \mathbb{P}(P)  \right\}.
		\label{AGGTSS}
	\end{equation}
	Note that this problem is equivalent to a TSSP with a discrete distribution of $|\mathcal{P}|$ scenarios for the uncertainty space. Moreover, this approach enables us to generate algorithms to obtain exact optimal solutions for general TSSPs.
	
	The remainder of this paper is organized as follows. \Cref{sec:LitRev} reviews the literature concerning the APM for discrete TSSPs  and other approaches to solve this problem.  \Cref{sec:GAPM} develops the generalized adaptive partition-based method (GAPM), with the main mathematical results to validate this approach.  \Cref{sec:algorithm} discusses the details of the algorithms that are implemented for two well-known stochastic programming problems in \Cref{sec:experiments}. Finally, concluding remarks are presented in \Cref{sec:conclusions}.

	\section{Literature Review}\label{sec:LitRev}

	In past decades, researchers have developed solution strategies for multiple stochastic optimization problems. However, the majority of studies start from the deterministic equivalent formulation to obtain alternative models that are more tractable in algorithmic terms, one of the most studied and utilized problems are two-stage stochastic problems. In their seminal paper, Kleywegt et al.~\cite{kleywegt2002sample} show that any TSSP formulation can be approximated by solving Problem \cref{TSSP} for a discrete set of samples of $\xi$ from the original probability space $\Omega$: they called this result the sample average approximation method. A key fact from the paper is that good approximations require a large number of scenarios to guarantee an $\epsilon$-optimal solution. Since then, most of the research on this problem has been focused on solving large-scale instances of discrete TSSPs with many scenarios. 

	A common and widely studied approach is to decompose TSSPs via the block structure of the scenario formulation. The most classic approach is called Benders decomposition (or the L-Shaped method as its stochastic variant~\cite{van1969shaped}). 

Most of the improvements of this approach focused on reducing the algorithm instability, such as the case of regularized decomposition~\cite{ruszczynski1986regularized}, level decomposition~\cite{lemarechal1995new,zverovich2012computational} and inexact bundle methods~\cite{oliveira2011inexact,wolf2014applying}. Recent developments with respect to Benders are proposed in~\cite{rahmaniani2017benders,rahmaniani2018accelerating,rahmaniani2019asynchronous}, which primarily explore how to accelerate and parallelize the technique, and \cite{angulo2016improving,ryan2016scenario}, which consider how to address integer problems.

	Other decomposition methodologies include stochastic decomposition~\cite{higle1991stochastic}, progressive hedging~\cite{rockafellar1991scenarios, watson2011progressive} and stochastic dual dynamic programming~\cite{pereira1991multi} for the case of multistage stochastic problems.
	
	A different approach was developed based on the general decomposition method proposed by Bienstock \& Zuckerberg~\cite{bienstock2010solving,munoz2018study}. Espinoza \& Moreno~\cite{espinoza2014primal} introduced an algorithm based on this decomposition method to minimize risk measures in linear programs. This idea was later extended by Song \& Luedtke~\cite{song2015adaptive} to general TSSPs with discrete distributions, where the term \emph{adaptive partition-based method} was coined. These studies have been extended recently by combination with Benders decomposition~\cite{pay2017partition}, level decomposition~\cite{pay2017partition,van2017adaptive}, and new extensions have been made to multi-stage stochastic problems~\cite{siddig2019adaptive}.

	As mentioned previously, most of the recent developments are oriented to the discrete case, relying on approximation by samples of continuous probability distributions for uncertain parameters. Exact methods for TSSPs with nondiscrete distributions are scarce, and they focus mostly on particular problems and distributions that can be reformulated in a more tractable manner. Other general techniques for these problems include~\cite{barika2013two}, which introduces equivalent linear and nonlinear formulations for TSSPs with simple recourse according to the probability distributions of random parameters, and~\cite{crainic2018reduced} which poses a methodology that benefits from the reduced cost of duality and sensitivity analysis to fix the correct values of some variables in the stochastic program, thereby reducing the size of the original problem.

	To the best of our knowledge, this paper is one of the first exact methods based on linear programming intended to deal with TSSPs with general continuous distributions for the stochastic parameters.
		
	\section{Generalized Adaptive Partition-based Method}\label{sec:GAPM}

  We propose a methodology which benefits from a structure shared by aggregated and atomized subproblems, which latter allows us to derive conditions such that the scenarios (either finite or infinite number of them) belonging to a certain element $P$, yield the same expected value of optimal solutions, as if we solve the aggregated Problem~\cref{secondst_agg}.

	\subsection{Relations between atomized and aggregated subproblems}
	
	As a first step, we define the relation between subproblems~\cref{secondst} and aggregated subproblems~\cref{secondst_agg}. Indeed, \cref{feasibility_solutions} shows how a feasible solution of \cref{secondst_agg} can be constructed using information from the optimal solution of \cref{secondst}.
	
	\begin{lemma}
		\label{feasibility_solutions}
		Let $\bar{x}\in \mathcal{X}$ and $P\subseteq \Omega$, and let $\hat{y}^\xi$ be the set of optimal solutions of $\mathcal{Q}(\bar{x},\xi)$ for $\xi \in P$. Then, $\hat{y}^P:= \mathbb{E}\big[ \hat{y}^\xi|P\big]$ is a feasible solution for $\mathcal{Q}(\bar{x}, \mathbb{E}\left[\xi|P\right])$.
	\end{lemma}
	
	\begin{proof}

		Since $W\hat{y}^\xi = h^\xi - T^\xi \bar{x}$ for every $\xi\in P$,  
		\begin{eqnarray*}
			\int_\Omega W \hat{y}^\xi \ d\mathbb{P}(\xi|P) &=& \int_\Omega \left[h^\xi - T^\xi \bar{x}\right] \ d\mathbb{P}(\xi|P)\\
			W \int_\Omega \hat{y}^\xi \ d\mathbb{P}(\xi|P) & = & \int_\Omega h^\xi \ d\mathbb{P}(\xi|P) - \left(\int_\Omega T^\xi  \ d\mathbb{P}(\xi|P)\right) \bar{x}\\
			W \mathbb{E}\left[\hat{y}^\xi|P\right] &=& \mathbb{E}\left[h^\xi|P\right] - \mathbb{E}\left[T^\xi|P\right]\bar{x}.
		\end{eqnarray*}
Hence, $\hat{y}^P$ is a feasible solution for $\mathcal{Q}(\bar{x}, \mathbb{E}\left[\xi|P\right])$.
	\end{proof}

	Since these second-stage subproblems consider only continuous variables, we can introduce a dual formulations for subproblems \cref{secondst} and \cref{secondst_agg}, respectively,
	
	\begin{equation}
	\mathcal{Q}^D(x,\xi) := \max \left\lbrace (h^\xi-T^\xi x)^\top \lambda^\xi ~|~ W^\top \lambda \leq q \right\rbrace 
	\label{Dualsp}
	\end{equation}
	and
	\begin{equation}
	\mathcal{Q}^D\left(x,P\right) := \max \left\lbrace (h^P-T^P x)^\top \lambda^P ~|~ W^\top \lambda^P \leq q \right\rbrace.
	\label{secondst_agg_dual}
	\end{equation}
Indices $\xi$ and $P$ on dual variable $\lambda$  distinguish between atomized and aggregated subproblems.

Similarly to the primal case, we can construct a feasible solution for problem \cref{secondst_agg_dual} based on the optimal solutions of \cref{Dualsp}.
	
	\begin{lemma}
		\label{dual_feas_sol}
		Let $\bar{x}\in \mathcal{X}$ and $P\subseteq \Omega$, and let $\hat{\lambda}^\xi$ be the optimal solution of problem $\mathcal{Q}^D(\bar{x},\xi)$ for $\xi \in P$. Then, $\hat{\lambda}^P:= \mathbb{E}\big[ \hat{\lambda}^\xi|P\big]$ is a feasible solution for $\mathcal{Q}^D(\bar{x},P)$
	\end{lemma}
	
	\begin{proof}

		Since $W^\top \hat\lambda^\xi \leq q$ for all $\xi\in P$, 
		\[W^\top \hat{\lambda}^P = W^\top \int_\Omega \hat\lambda^\xi d\mathbb{P}(\xi|P) = \int_\Omega W^\top\hat\lambda^\xi d\mathbb{P}(\xi|P) \leq  \int_\Omega q  \cdot d\mathbb{P}(\xi|P) = q. \]
		\label{prop2conc}
Hence, $\hat\lambda^P$ is a feasible solution for problem $\mathcal{Q}^D(\bar{x},P)$
		whenever set $P$ has positive measure.
	\end{proof}
	
	\subsection{Construction of an optimal partition}

	The previous framework provides the set of tools necessary to set $\mathcal{Q}(\bar{x},\mathbb{E}\left[\xi|P\right])$ as a lower bound of $\mathbb{E}\left[\mathcal{Q}(\bar{x},\xi)|P\right]$; furthermore, we identify the conditions on $\mathcal{P}$ to make this bound tight.
	
	\begin{proposition}
		Let $\bar{x}\in\mathcal{X}$ and $P\subseteq \Omega$, such that $\mathcal{Q}(\bar{x},\xi)$ is feasible for all $\xi\in P$, and let $\hat{\lambda}^\xi$ be its dual optimal solutions. If $\hat{\lambda}^\xi$ for $\xi\in P$ satisfies
		\begin{subequations}
			\label{HypoLemmaGen}
			\begin{eqnarray}
				\label{HypoLemmaGen1}
				\Big(\mathbb{E}\big[ h^\xi | P \big]\Big) ^\top \Big( \mathbb{E}\big[ \hat{\lambda}^\xi | P \big]\Big) &=& \mathbb{E}\left[ \left.h^\xi\right.^\top \hat{\lambda}^\xi \Big|  P \right]\\
				\label{HypoLemmaGen2}
				\bar{x}^\top\Big(\mathbb{E}\big[ T^\xi | P \big]^\top \mathbb{E}\big[ \hat{\lambda}^\xi | P \big]\Big) &=& \bar{x}^\top \mathbb{E}\left[ \left.T^\xi\right.^\top \hat{\lambda}^\xi \Big| P \right]	
			\end{eqnarray}
		\end{subequations}
	then, 
		\[ \mathcal{Q}\big(\bar{x},\mathbb{E}\left[\xi|P\right] \big) = \mathbb{E}\big[ \mathcal{Q}(\bar{x},\xi) | P \big] \]
		\label{GenLemma}
	\end{proposition}

	\begin{proof}
		We first note that $\mathcal{Q}(\bar{x},\xi)$ is a convex function on $\xi$. Let us consider the problem $f(\xi)=\min\{q^\top y|Wy = \xi\}$. Now, we can take random values of $b$ namely $b_1$ and $b_2$, with $y_1^\star$ and $y_2^\star$ as their respective optimal solutions. If $b_\beta$ is a convex combination of $b_1$ and $b_2$, then $\beta y_1^\star+(1-\beta)y_2^\star$ is a feasible solution of $f(b_\beta)$. If $y_\beta^\star$ is the optimal solution of $f(b_\beta)$, we can build
		\begin{eqnarray*}
			f(b_\beta)&=&q^\top y_\beta^\star \leq q^\top(\beta y_1^\star+(1-\beta)y_2^\star)\\
			&=& \beta q^\top y_1^\star + (1-\beta) q^\top y_2^\star = \beta f(b_1) + (1-\beta)f(b_2)
		\end{eqnarray*}	
		which verifies $\mathcal{Q}(\bar{x},\xi)$ is convex on $\xi$. Hence, by applying Jensen's inequality, we obtain $\mathcal{Q}(\bar{x},\mathbb{E}\left[\xi|P\right]) \leq \mathbb{E}\left[\mathcal{Q}(\bar{x},\xi)|P\right]$.
		
		On the other hand, according to \cref{dual_feas_sol}, we know that $\hat{\lambda}^P:= \mathbb{E}[\hat \lambda^\xi|P]$ is a feasible solution of $\mathcal{Q}^D(\bar{x},P)$; thus,
		\begin{align*}
		\mathcal{Q}^D\big(\bar{x},P \big) = \mathcal{Q}(\bar{x},\mathbb{E}\left[\xi|P\right]) & \geq \Big( \mathbb{E}\big[ h^\xi | P \big] - \mathbb{E}\big[ T^\xi | P \big]\bar{x} \Big)^\top  \Big(\mathbb{E}\big[ \hat{\lambda}^\xi |  P \big]\Big)\\ 
		& = \Big(\mathbb{E}\big[ h^\xi | P \big]\Big)^\top \Big( \mathbb{E}\big[ \hat{\lambda}^\xi | P \big]\Big) - \bar{x}^\top\Big(\mathbb{E}\big[ T^\xi | P \big]\Big)^\top \Big( \mathbb{E}\big[ \hat{\lambda}^\xi | P \big]\Big)
		\end{align*}
		
		Since $\hat{\lambda}^P$ satisfies conditions \cref{HypoLemmaGen}, by means of the linearity of the expectation, we obtain
		\begin{align*}
		\mathcal{Q}(\bar{x},\mathbb{E}\left[\xi|P\right]) & \ge \Big(\mathbb{E}\left[ \left(h^\xi\right)^\top \hat{\lambda}^\xi \Big| P \right]\Big) - \bar{x}^\top\Big(\mathbb{E}\left[ \left(T^\xi\right)^\top \hat{\lambda}^\xi \Big| P \right]\Big) \\
		& =  \mathbb{E}\big[ \mathcal{Q}(\bar{x},\xi) | P \big]
		\end{align*}
	\end{proof}

	\begin{theorem} \label{cor:partition_implies_original}
		Let	$x^*$ be an optimal solution of problem
	\[		\min_{x\in\mathcal{X}} \left\{ c^\top x +  \sum_{P\in\mathcal{P^*}}\mathcal{Q}\left(x,\mathbb{E}\left[\xi|P\right]\right) \cdot \mathbb{P}(P)  \right\}  \]
		where $\mathcal{P}^*$ is a numerable partition of $\Omega$ such that for each $P\in\mathcal{P}^*$, the optimal dual variables of $\mathcal{Q}(x^*,\xi)$ for $\xi\in P$ satisfy conditions \cref{HypoLemmaGen}.  Then, $x^*$ is also an optimal solution of problem 
		\[	\min_{x\in\mathcal{X}} \left\{ c^\top x + \mathbb{E} \left[ \mathcal{Q}(x,\xi)\right] \right\}. \]
	\end{theorem}
	\begin{proof} 
		By the laws of total expectation, we know that for any numerable partition $\mathcal{P}$ of $\Omega$, 
		\[ \mathbb{E} \left[ \mathcal{Q}(x,\xi)\right] = \sum_{P\in\mathcal{P}}\mathbb{E}\left[\mathcal{Q}\left(x,\xi\right)|P\right] \cdot \mathbb{P}(P) \]
		In particular, for $x^*$ and $\mathcal{P}^*$, according to \cref{GenLemma}, we obtain
		\[	 c^\top x^* + \mathbb{E} \left[ \mathcal{Q}(x^*,\xi)\right] =  c^\top x^* +  \sum_{P\in\mathcal{P}^*}\mathcal{Q}\left(x^*,\mathbb{E}\left[\xi|P\right]\right) \cdot \mathbb{P}(P),  \]
		hence,
		\[	 \min_{x\in\mathcal{X}} \left\{c^\top x + \mathbb{E} \left[ \mathcal{Q}(x,\xi)\right] \right\}  \leq \min_{x\in\mathcal{X}} \left\{  c^\top x +  \sum_{P\in\mathcal{P}^*}\mathcal{Q}\left(x,\mathbb{E}\left[\xi|P\right]\right) \cdot \mathbb{P}(P)\right\}.  \]

On the other hand, if $\hat{x}$ is the optimal solution of Problem~\cref{TSSP}, then 
\begin{align*}
	c^\top \hat{x} + \mathbb{E} \left[ \mathcal{Q}(\hat{x},\xi)\right] 
	&=     c^\top \hat{x} +  \sum_{P\in\mathcal{P}^*}\mathbb{E}\left[\mathcal{Q}\left(\hat{x},\xi\right)|P\right] \cdot \mathbb{P}(P) \\
	&\geq  c^\top \hat{x} +  \sum_{P\in\mathcal{P}^*}\mathcal{Q}\left(\hat{x},\mathbb{E}\left[\xi|P\right]\right) \cdot \mathbb{P}(P) \\
	&\geq  \min_{x\in\mathcal{X}} \left\{  c^\top x +  \sum_{P\in\mathcal{P}^*}\mathcal{Q}\left(x,\mathbb{E}\left[\xi|P\right]\right) \cdot \mathbb{P}(P)\right\}, 
\end{align*}
where the first equality is true by the laws of total expectation and the second inequality is given by Jensen's inequality and the convexity of $\mathcal{Q}(\hat{x},\cdot)$.
\end{proof}

Note that this partition always exists, as presented in the following corollary.

\begin{corollary} \label{cor:original_implies_partition}
	If	$x^*$ is an optimal solution of problem
\[			\min_{x\in\mathcal{X}} \left\{ c^\top x + \mathbb{E} \left[ \mathcal{Q}(x,\xi)\right] \right\} \]
then there exists a finite partition $\mathcal{P}^*$ of $\Omega$ such that

	\[ c^\top x^* + \mathbb{E} \left[ \mathcal{Q}(x^*,\xi)\right] 
	=  c^\top x^* +  \sum_{P\in\mathcal{P}^*}\mathcal{Q}\left(x^*,\mathbb{E}\left[\xi|P\right]\right) \cdot \mathbb{P}(P) \]
	\end{corollary}

	\begin{proof}
		Note that the dual feasible solutions $\lambda$ of $Q(x^*,\xi)$ must satisfy $W^\top \lambda \leq q$.
Hence, for each $\xi\in\Omega$, we have an associated  extreme point of $W^\top \lambda \leq q$ that is an optimal dual solution of $Q(x^*,\xi)$. This result induces a finite partition $\mathcal{P^*}$ of $\Omega$ such that all $\xi\in P$ have the same dual optimal solution of $Q(x^*,\xi)$. Hence, the values all satisfy the conditions of \cref{GenLemma}, and 
\begin{align*}
	c^\top x^* + \mathbb{E} \left[ \mathcal{Q}(x^*,\xi)\right] 
	&=     c^\top x^* +  \sum_{P\in\mathcal{P}^*}\mathbb{E}\left[\mathcal{Q}\left(x^*,\xi\right)|P\right] \cdot \mathbb{P}(P) \\
	&=  c^\top x^* +  \sum_{P\in\mathcal{P}^*}\mathcal{Q}\left(x^*,\mathbb{E}\left[\xi|P\right]\right) \cdot \mathbb{P}(P).
\end{align*}
\end{proof}

We finish this section by noting some differences regarding the original APM proposed in~\cite{song2015adaptive}.  The most relevant aspect of this  proof is the possibility to extend it to the case of continuous probability space.  However, this proof also shows that the method can be applied to any set of primal variables $x\in\mathcal{X}$. Moreover, the condition proposed in the original paper to aggregate scenarios in such that all dual variables $\lambda^\xi$ for $\xi\in P$ on each subset $P\subset\Omega$ must have the same value. This is a particular case which satisfies the conditions of \cref{GenLemma} by means of the linearity of the expected value. Finally, \cite{song2015adaptive} establishes that this criterion is required to have the equality between the value of the aggregated problem and the expected value of the atomized subproblems. Nonetheless, the presented conditions of \cref{GenLemma} provides a framework where less demanding conditions might be applied to aggregate/disaggregate scenarios, e.g., degenerated subproblems with multiple optimal dual solutions.

\section{Algorithm implementation}\label{sec:algorithm}

The idea of the method is to iteratively converge to a partition $\mathcal{P}$ satisfying the conditions of \cref{GenLemma}. Initially, we start with a trivial partition ($\mathcal{P}=\{\Omega\}$) and split the partition based on the duals of the subproblems. This split procedure is problem dependent. At each iteration, the algorithm provides a lower bound (the optimal value of the aggregated problem) and, potentially, an upper bound. The upper bound can be computed by solving the expected value of the subproblem, which is easy to compute in the discrete case by solving the subproblem for each scenario independently. However, the computation can be difficult for continuous distributions. The algorithm is presented in \cref{alg:GAPM}.

\begin{algorithm}[tbhp]
	\begin{algorithmic}[1]

		\REQUIRE {A stopping threshold $\epsilon$ and an initial partition $\mathcal{P}^{(0)}$ of $\Omega$} 
	\ENSURE {An optimal partition $\mathcal{P^*}$ to attain the optimal solution $x^\star$ of problem \cref{TSSP}}

	\STATE Set $t:=0$, $z_L^{(0)}:=-\infty$ and $z_U^{(0)}=\infty$.
	\LOOP
		\STATE $t=t+1$
		\STATE Solve Problem \cref{AGGTSS} for partition $\mathcal{P}^{(t)}$ and assign its optimal value to the lower bound $z_L^{(t)}$ and its optimal solution to $\bar{x}^{(t)}$
		\STATE\label{alg:ub} If possible, compute the upper bound $z_U^{(t)}=c^\top \bar{x}^{(t)} + \mathbb{E}[\mathcal{Q}(\bar{x}^{(t)} ,\xi)]$
		\STATE Solve subproblems \cref{secondst_agg} for every $P\in \mathcal{P}^{(t)}$
		\IF{$z_U^{(t)} - z_L^{(t)} < \epsilon$}
				\STATE \textbf{exit} 
		\ENDIF	
		\STATE\label{alg:disagg} Run disaggregation procedure to split uncertain region and obtain $\mathcal{P}^{(t+1)}$
    \label{alg:compDuals}\IF{ $\mathcal{P}^{(t+1)}$ satisfies \cref{HypoLemmaGen}}
      \STATE \textbf{exit}
		\ENDIF
		
	\ENDLOOP
	
	\RETURN optimal solution $x^\star:=\bar{x}^{(t)}$, optimal partition $\mathcal{P}^\star:=\mathcal{P}^{(t)}$ and  optimal value $z_L^{(t)}$
	\end{algorithmic}
	\caption{An iterative implementation of GAPM}
	\label{alg:GAPM}
\end{algorithm}

There are two key steps in the GAPM implementation proposed in \cref{alg:GAPM}. First, to execute \cref{alg:disagg} correctly, we require an additional procedure to split the uncertain region in an appropriate manner; this step is fine-tuned according to the structure of subproblems~\cref{secondst}. We discuss this point in the computational experiments in \Cref{sec:experiments}.
On the other hand, if an upper bound $z_U^{(t)}$ cannot be computed, on \cref{alg:compDuals} we can still check \cref{GenLemma} or even compare composition of consecutive partitions $\mathcal{P}^{(t)}$ and $\mathcal{P}^{(t+1)}$. 

\section{Numerical experiments}\label{sec:experiments}

Since the fundamental novelty of our proposal arises when stochastic parameters have continuous probability distributions, the computational experiments are designed to enlighten algorithmic behaviour on two problems from classic literature with this type of uncertainty. 
For the case of a discrete distribution, we refer the reader to the papers presented in the literature review.

We have divided the computational experiment into two parts. First, we discuss the implementation and results for a classic problem from the stochastic programming literature, the LandS instance, wherein uncertainty is presented in the RHS coefficients. The second problem is the TSSP reformulation of conditional value-at-risk (CVaR) minimization, where the uncertainty appears in the technological coefficient of the first-stage variables $x$.  Both problems have well-defined structures that are useful to define the procedure to split the uncertainty space $\Omega$ at each iteration of the algorithm.

\subsection{Energy planning problem - LandS}

$\mathrm{LandS}$, a classic problem in stochastic programming that is studied for academic purposes, was originally proposed in~\cite{landS}. LandS in an energy planning investment problem, where the goal is to decide the capacities of four new plants while minimizing allocation and operational costs. The set of power plants are supposed to meet uncertain demand of three different electric modes. In the fist stage, some minimum capacities and budget constraints must be satisfied; during the second stage, energy is distributed according to the realization of the uncertain demands. The mathematical formulation is as follows:

\begin{equation}
	\min\limits_{x \geq 0}  \left\{ \sum\limits_{i\in  \mathbb{I}}  c_ix_i + \mathbb{E}\left[ \mathcal{Q} (x,\xi)\right]  :  \sum\limits_{i\in \mathbb{I}}  x_i \geq m, 
	\sum\limits_{i\in \mathbb{I}} c_ix_i \leq b \right\}
\end{equation}
where 
\begin{subequations}
\begin{align}
\mathcal{Q}(x,\xi):=\min\limits_{y \geq 0} ~ & \sum\limits_{i\in \mathbb{I}}\sum\limits_{j\in \mathbb{J}}f_{ij}y_{ij} \\
\text{s.t}~
&\sum\limits_{j\in \mathbb{J}} y_{ij} \leq x_i,\quad \forall~ i\in \mathbb{I}  \label{LandS_capacityconst}\\
&\sum\limits_{i\in \mathbb{I}} y_{ij} \geq d_j^\xi ,\quad \forall~ j\in \mathbb{J}\label{LandS_demandconst}
\end{align}
\label{LandS_sub}
\end{subequations}

The original problem sets up an uncertain demand for $d_1^\xi$ with three scenarios: 3, 5 or 7 units. In this experiment, we assume that $d_1^\xi$ follows a uniform distribution in the interval [3,7], following the ideas from~\cite{linderoth2006empirical}. Remaining demands are considered to be deterministic.

To split the uncertainty space $\Omega$ and compute an upper bound for the optimal value of the problem, we introduce the dual of $\mathcal{Q}(\hat{x},\xi)$ given by
\begin{equation}
\begin{aligned}
\mathcal{Q}^D(\hat{x},\xi) := \max\limits_{\nu,\mu \geq 0}  \sum\limits_{j\in \mathbb{J}} \mu_j d_j^\xi ~ & -   \sum\limits_{i\in \mathbb{I}} \nu_i \hat{x}_i\\
\mu_j - \nu_i \leq f_{ij} & \qquad \forall i\in \mathbb{I}, j\in \mathbb{J}
\end{aligned}
\end{equation}
where $\nu$ and $\mu$ correspond to the dual variables of constraints \cref{LandS_capacityconst} and \cref{LandS_demandconst}, respectively.
Then, given an optimal solution of the subproblem for a given value of $d_1^\xi$, we can use sensitivity analysis to compute a neighbourhood around $d_1^\xi$ in which the dual optimal variables do not change. Moreover, $\mathcal{Q}^D({x},\xi)$ is a non-decreasing piecewise linear function on $d_1^\xi$, so the upper bound of \cref{alg:ub} is easy to compute.

In our experiment, we start with $\mathcal{P}^{(0)}=\{[3,7]\}$ and in each iteration, the partition is refined by dividing the corresponding elements of $\mathcal{P}$, utilizing  the segment extremes of piecewise linear function $\mathcal{Q}^D(\bar{x}^{(t)},\xi)$.

\begin{figure}[tbhp]
	\centering
	\begin{subfigure}[b]{0.50\textwidth}
		\includegraphics[width=\linewidth]{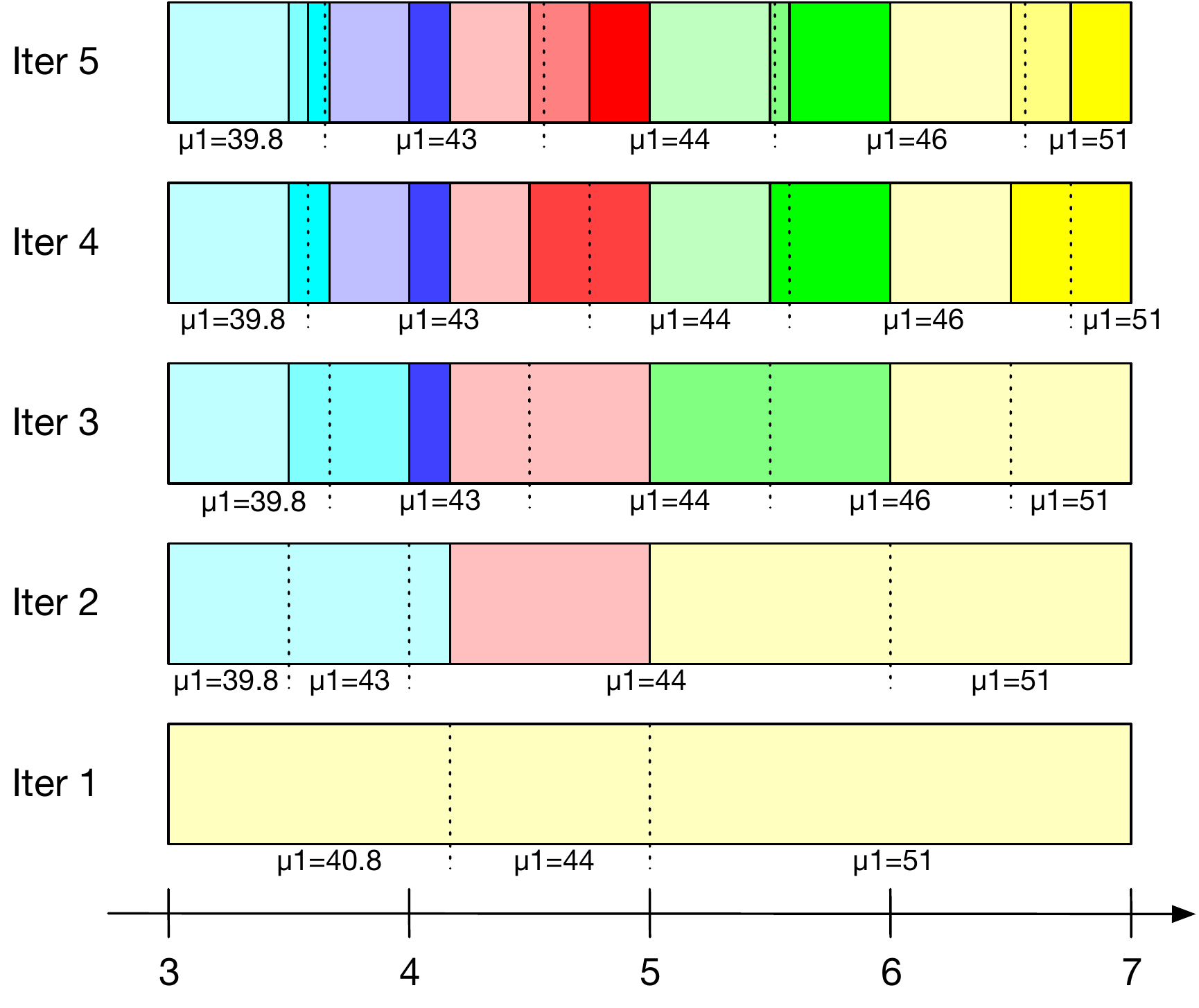}
		\caption{Partition of $\Omega$ in each iteration}\label{subfig:landS_partition}
	\end{subfigure} 
	~%
	\begin{subfigure}[b]{0.48\textwidth}
	\begin{tabular}{|c||c|c|c|c|}
	Iter & $x_1$ & $x_2$ & $x_3$ & $x_4$\\ \hline
1 & 0.833 & 3.000 & 4.167 & 4.000 \\
2 & 2.500 & 3.000 & 3.500 & 3.000 \\
3 & 1.833 & 4.000 & 3.667 & 2.500 \\
4 & 2.000 & 4.167 & 3.583 & 2.250 \\
5 & 1.917 & 4.083 & 3.625 & 2.375 \\
6 & 1.875 & 4.042 & 3.646 & 2.438 \\
\end{tabular} 

\begin{tabular}{|c||c|c|c|c||c|c|c|}
	Iter & LB & UB & Gap \\ \hline
1 & 378.667 & 382.711 & 1.0567\%\\
2 & 380.122 & 381.100 & 0.2567\%\\
3 & 380.601 & 380.844 & 0.0640\%\\
4 & 380.842 & 380.893 & 0.0007\%\\
5 & 380.843 & 380.856 & 0.0004\%\\
6 & 380.844 & 380.847 & 0.0002\%\\
\end{tabular}
\caption{Solution and objective values}
\end{subfigure}
\caption{Iteration details for the LandS example}\label{fig:landS}
\end{figure}

In \cref{fig:landS}, we show the resulting first six iterations of the algorithm.  Columns LB and UB present the current lower bound (objective value of the aggregated problem) and the upper bound (computed by the benefit of $\bar{x}^{(t)}$ and subproblem optimal dual variables), respectively. The column {Gap} shows the relative gap between the current solution and the best upper bound obtained so far.  After a few iterations, we obtain near-optimal solutions for the problem, with a gap close to the computational precision of the optimization software. In \cref{subfig:landS_partition}, we present the partition in each iteration (highlighted by different colours), as well as the segments  (dotted lines) obtained after carrying out the sensitivity analysis. The value under each segment corresponds to the dual variable of the stochastic demand constraint. Notably, these dual values do not change after iteration 3, but the extremes of the corresponding intervals change slightly in each iteration until converging to the optimal solution.

\subsection{Conditional value-at-risk linear problems}

A classic problem in risk optimization is to minimize the CVaR, which is a well-known risk measure satisfying the properties of coherency~\cite{artzner1999coherent}. 
In our case, we assume a linear problem, where the objective coefficients $\tilde{r}^\xi$ are random, and we minimize the $\text{CVaR}(x^\top{\tilde{r}^\xi})$ subject to linear constraints $Ax\leq b$. This problem (see~\cite{rockafellar2000optimization}) can be reformulated as
\[ \min_{x,\tau} \left\{ \tau + \frac{1}{\delta}\mathbb{E} \left[-x^\top\tilde{r}^\xi-\tau\right]^+ : Ax\leq b \right\} \]

In our context, $x$ and $\tau$ are the first-stage decisions, while the second-stage subproblem is
\[\mathcal{Q}( (x,\tau),\xi)) := (-x^\top\tilde{r}^\xi-\tau)^+ = \min \{z : z \geq -x^\top\tilde{r}^\xi-\tau, z\geq 0 \}\]
Let us note that the dual of $\mathcal{Q}( (x,\tau),\xi))$ has a single dual variable $\lambda$, and it can be formulated as
\[ \max_\lambda \left\{ (-x^\top\tilde{r}^\xi-\tau)\cdot \lambda : \lambda \leq 1, \lambda\geq 0 \right\}. \] 
Hence, the optimal solution of this dual problem is 
\[ \lambda^* = \begin{cases} 1 & \text{if } -\bar{x}^\top\tilde{r}^\xi - \bar{\tau}  \geq 0\\ 0 & \text{if not} \end{cases} \]
	In other words, there is a hyperplane separating $\Omega$, where the dual variables of the subproblem $\mathcal{Q}\left(  (\bar{x},\bar{\tau}), \xi \right)$ have the same value for a given pair $(\bar{x},\bar{\tau})$.

Therefore, from a partition $\mathcal{P}^{(t)}$ of $\Omega$, we can compute $r_P=\mathbb{E}[\tilde{r}^\xi|P]$m $p_P=\mathbb{P}(P)$ and solve the aggregated problem
\[\min_{x,\tau} \left\{ \tau + \frac{1}{\delta}\sum_{P\in\mathcal{P}^{(t)}}p_P \cdot z_P : Ax\leq b, %
	z_P \geq -x^\top r_P-\tau,   %
	z_P\geq 0 \ \forall P\in\mathcal{P}^{(t)} \right\} \]

	Given the optimal solution $(\bar{x}^{(t)},\bar{\tau}^{(t)})$ of this problem, we can split each $P\in\mathcal{P}^{(t)}$ into subsets $P'=P\bigcap \{\xi : -\bar{x}^{(t)\top} \tilde{r}^\xi \geq \bar{\tau}^{(t)} \}$ and $P''=P\bigcap \{\xi : -\bar{x}^{(t)\top}\tilde{r}^\xi \leq \bar{\tau}^{(t)} \}$ to obtain a new partition. 

	\subsubsection*{Case study:}
	For the computational test, we solve the classic portfolio problem, where $x$ represents the fraction of the portfolio assigned to each investment and the constraints of the first stage are $x^\top e=1, x\geq 0$, ensuring to invest the whole budget in non-negative fractions. Additionally, we assume that returns $\tilde{r}$ of each investment follow a multivariate normal distribution $\tilde{r}^\xi\leadsto \mathcal{N}(\mu, \Sigma)$ using historical data for stocks listed on the SP500, as in~\cite{vielma2008lifted} and~\cite{lagos2015restricted}.
 
	Note that in each iteration, given $(\bar{x}^{(t)},\bar{\tau}^{(t)})$, we can compute an upper bound for the problem expressed as
\begin{equation}
	\text{CVaR}_\delta (\bar{x}^{(t)\top}\tilde{r}^\xi) := \mu^\top \bar{x}^{(t)} + \tfrac{\sigma}{\delta} \phi(\Phi^{-1}(\delta)) 
\end{equation}
where $\sigma=\bar{x}^{(t)\top}\Sigma \bar{x}^{(t)}$ and $\phi$ and $\Phi$ are the standard normal p.d.f and standard normal quantile, respectively.

\begin{table}[tbhp]
	\caption{Results for the CVaR portfolio example}\label{tab:cvar2s0.1a}
	\centering
\begin{tabular}{|c||c|c|c||c||c|c|}
	Iter & LB & UB & Gap & $|\mathcal{P}^{(t)}|$ & $x_1$ & $x_2$\\ \hline
1 &-0.0702 & 0.7641 & 109.184\% & 1 & 0 & 1\\
2 & 0.0408 & 0.6054 & 93.2602\% & 2 & 1 & 0\\
3 & 0.3196 & 0.6054 & 47.2124\% & 4 & 1 & 0\\
4 & 0.3585 & 0.7641 & 40.7887\% & 6 & 0 & 1\\
5 & 0.4584 & 0.5104 & 10.1866\% & 9 & 0.59 & 0.41\\
6 & 0.5001 & 0.5222 & 2.0277\% & 14 & 0.7752 & 0.2248\\
7 & 0.5043 & 0.5095 & 1.0259\% & 20 & 0.6834 & 0.3166\\
8 & 0.5070 & 0.5082 & 0.2305\% & 27 & 0.6371 & 0.3629\\
9 & 0.5082 & 0.5082 & 0.0039\% & 34 & 0.6375 & 0.3625
\end{tabular} 
\end{table}

\begin{figure}[tbhp]
	\centering
	\begin{subfigure}[b]{0.31\textwidth}
		\includegraphics[width=\linewidth]{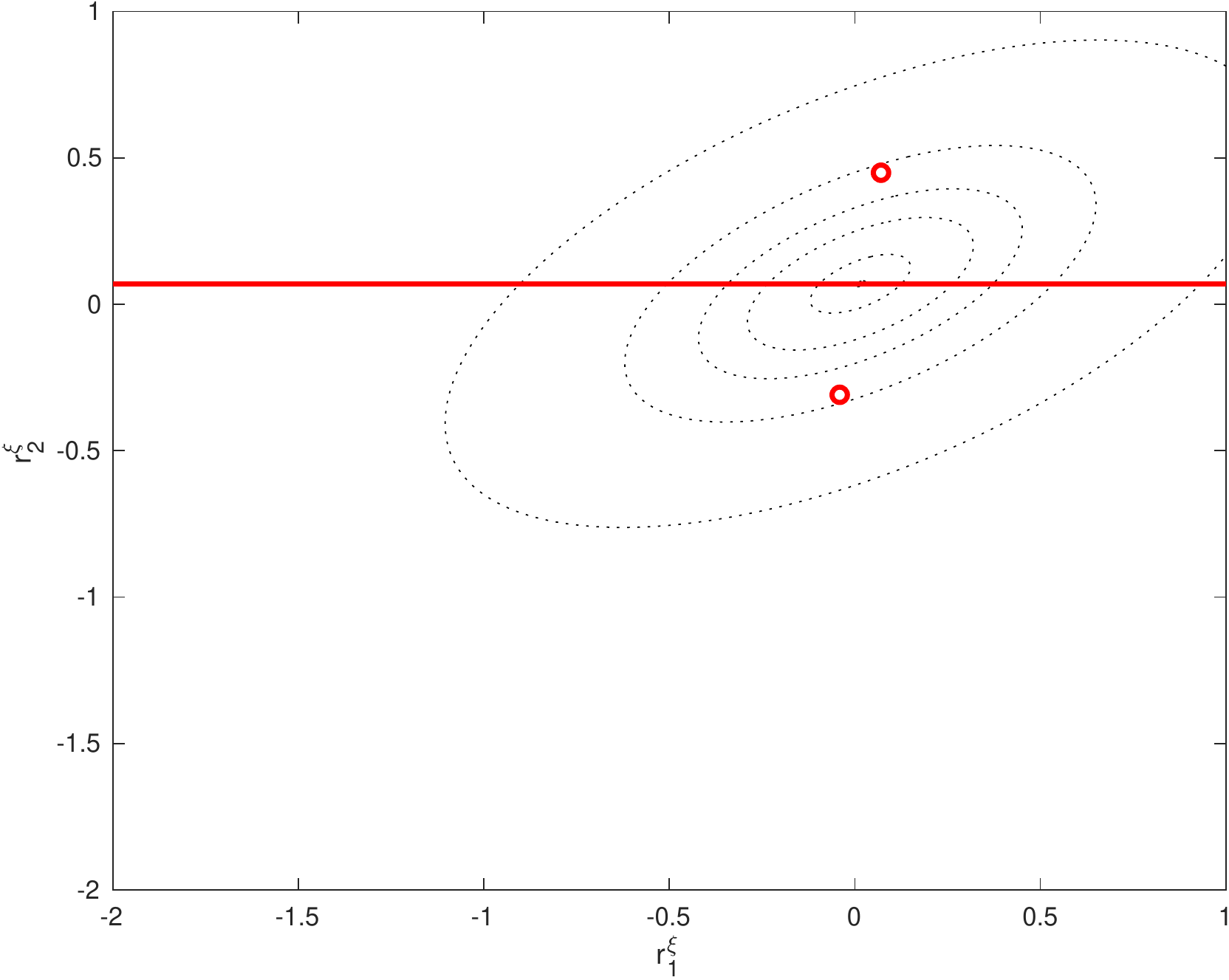}
		\caption{Iteration 1}
	\end{subfigure} 
	~%
	\begin{subfigure}[b]{0.31\textwidth}
		\includegraphics[width=\linewidth]{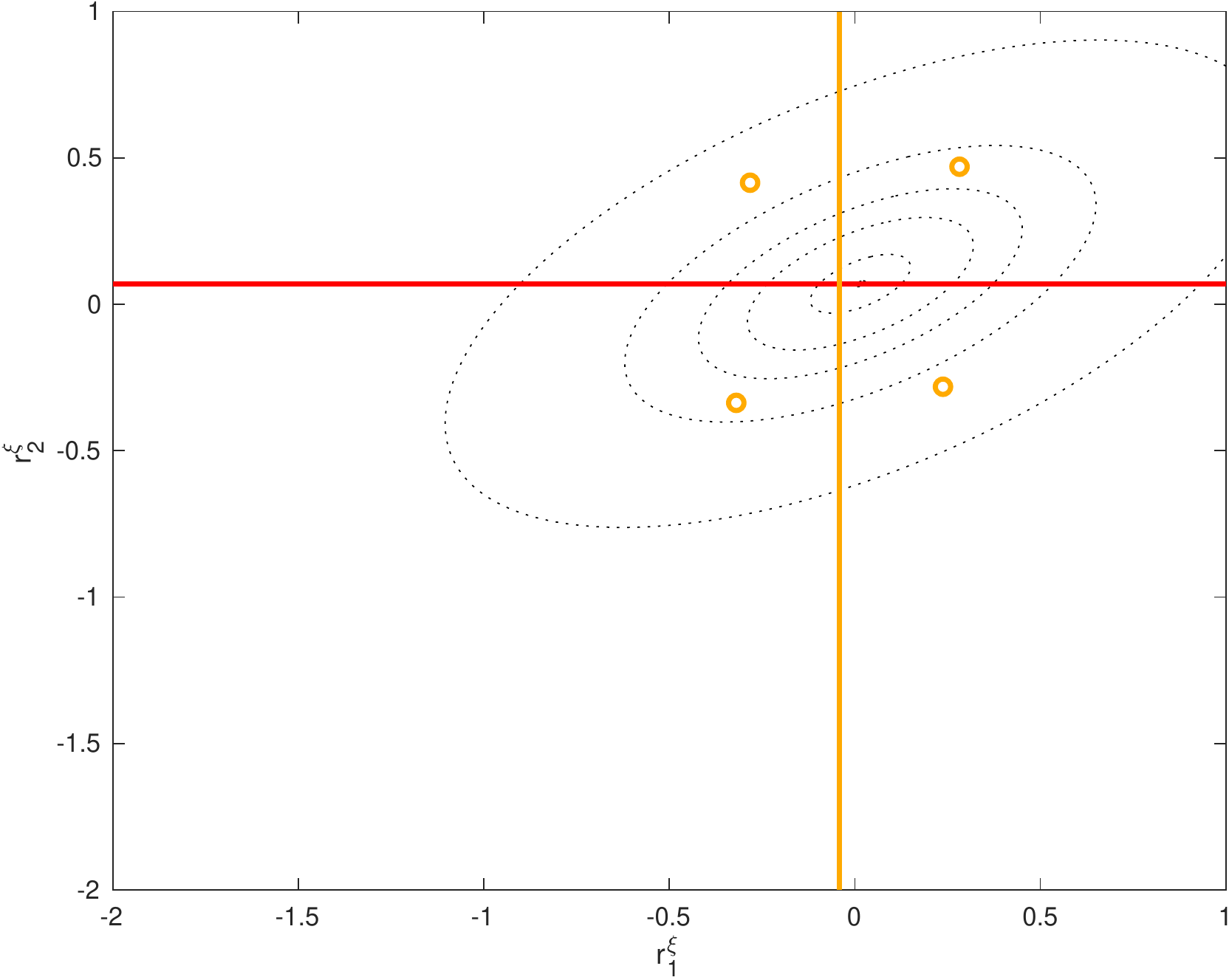}
		\caption{Iteration 2}
	\end{subfigure} 
	\begin{subfigure}[b]{0.31\textwidth}
		\includegraphics[width=\linewidth]{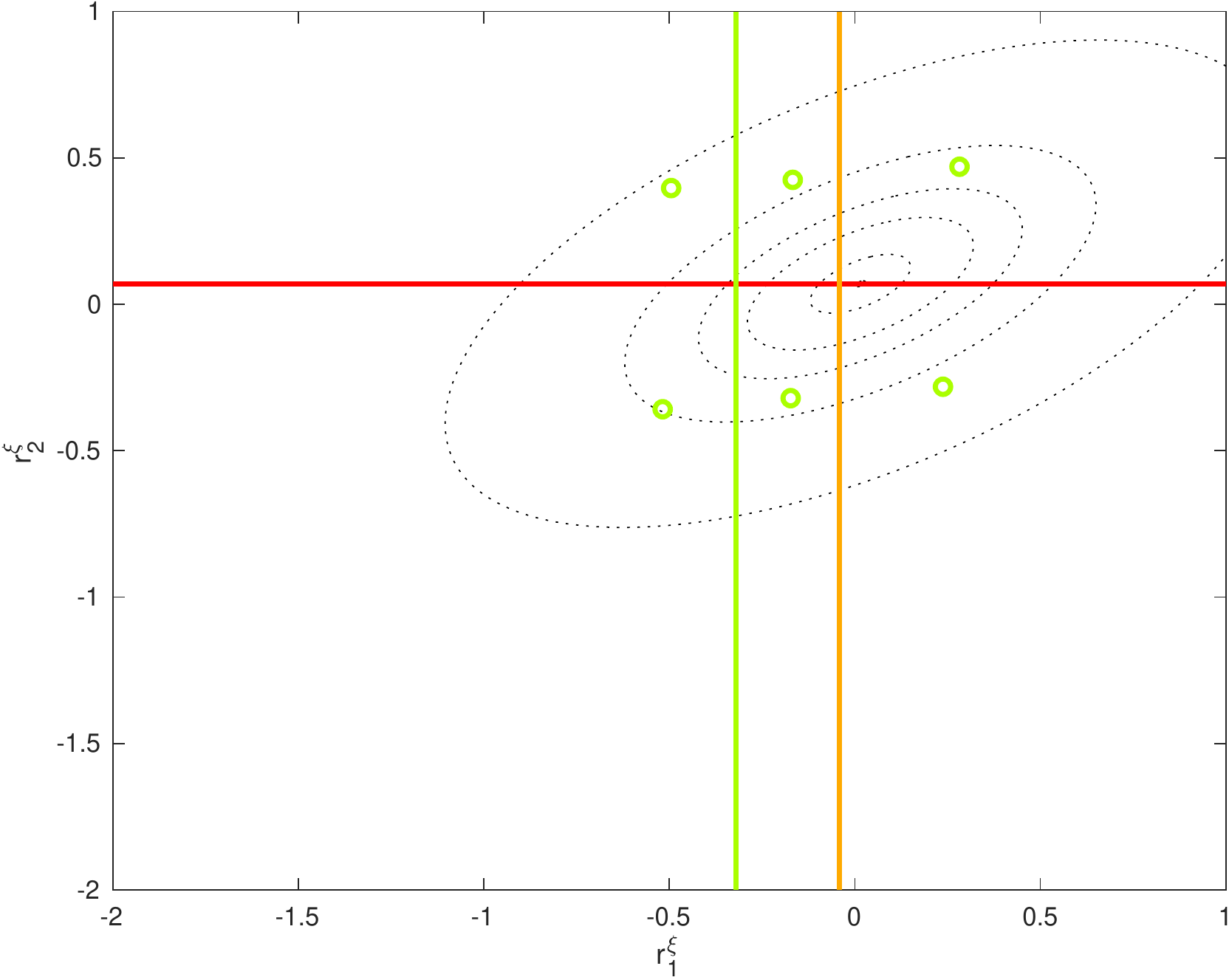}
		\caption{Iteration 3}
	\end{subfigure} 
	\\
	\begin{subfigure}[b]{0.31\textwidth}
		\includegraphics[width=\linewidth]{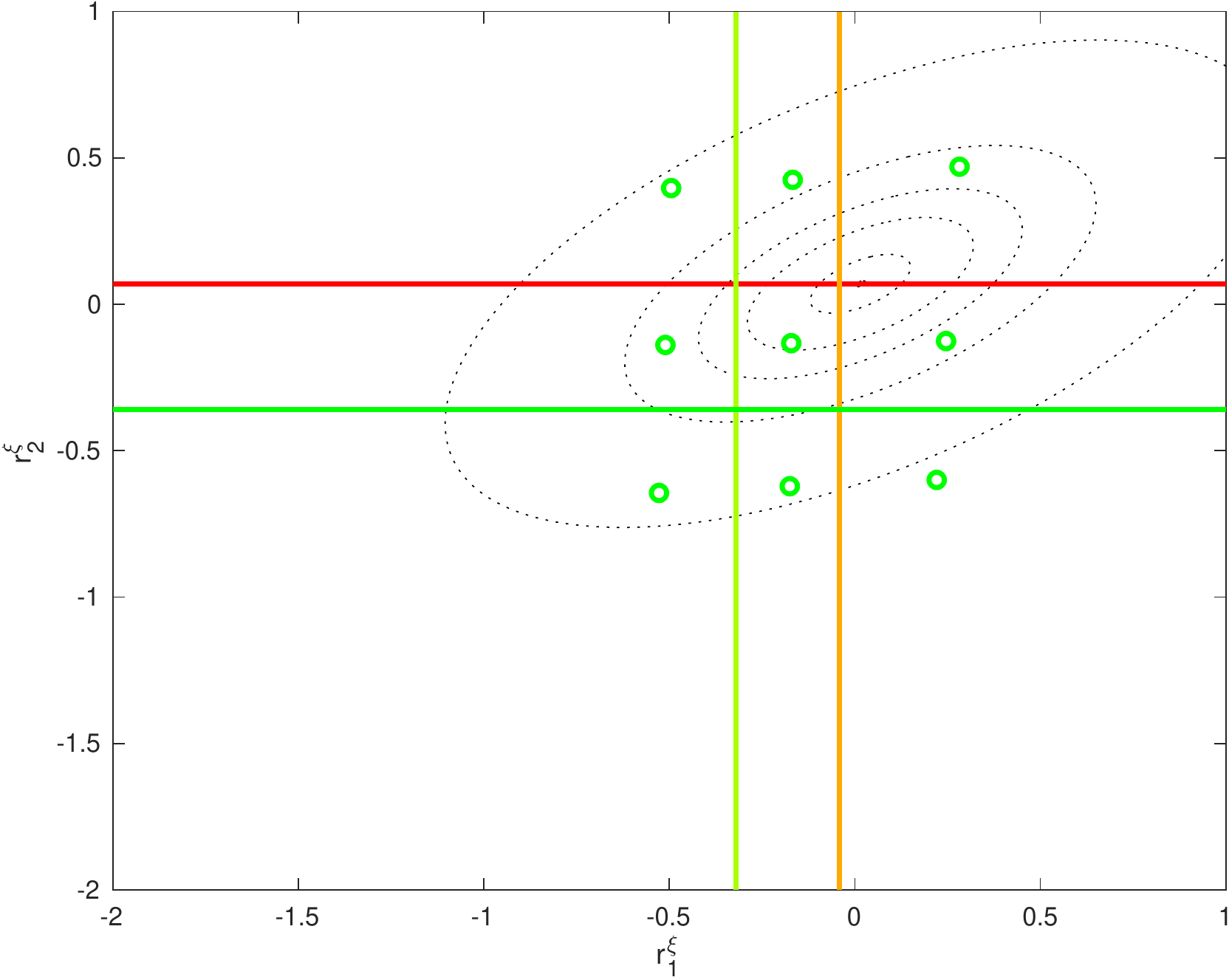}
		\caption{Iteration 4}
	\end{subfigure} 
	~%
	\begin{subfigure}[b]{0.31\textwidth}
		\includegraphics[width=\linewidth]{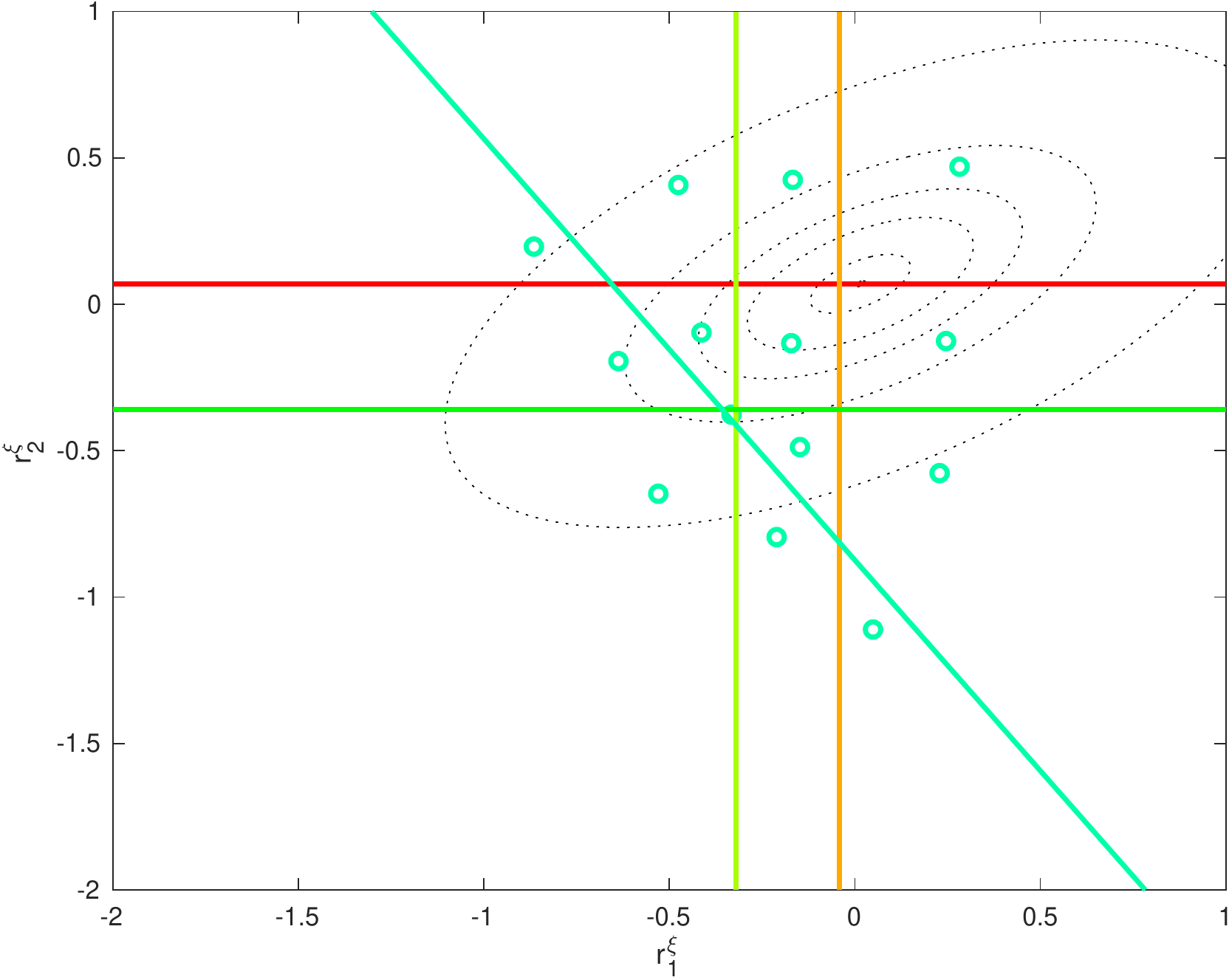}
		\caption{Iteration 5}
	\end{subfigure} 
	~%
	\begin{subfigure}[b]{0.31\textwidth}
		\includegraphics[width=\linewidth]{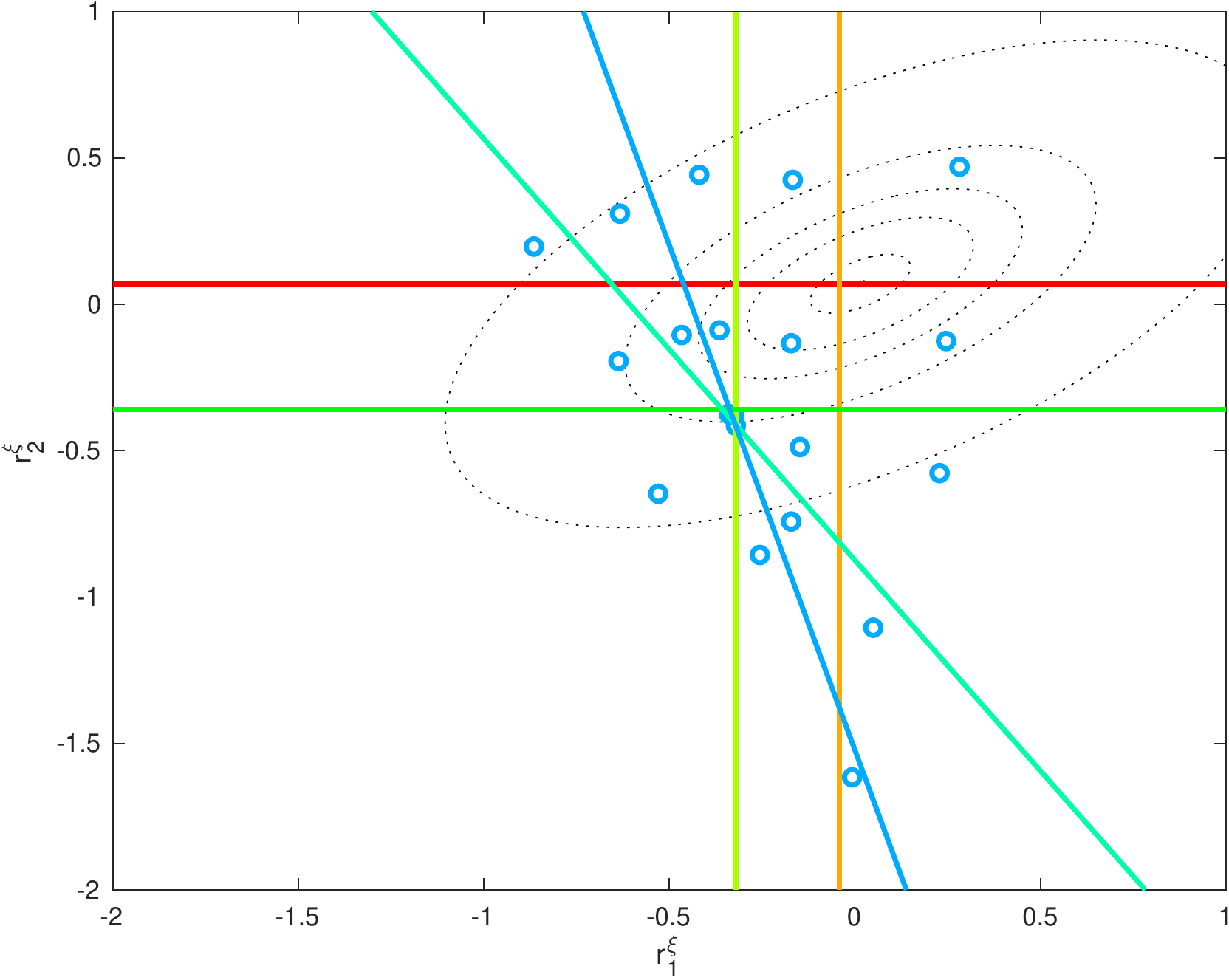}
		\caption{Iteration 6}
	\end{subfigure} 
	\\
	\begin{subfigure}[b]{0.31\textwidth}
		\includegraphics[width=\linewidth]{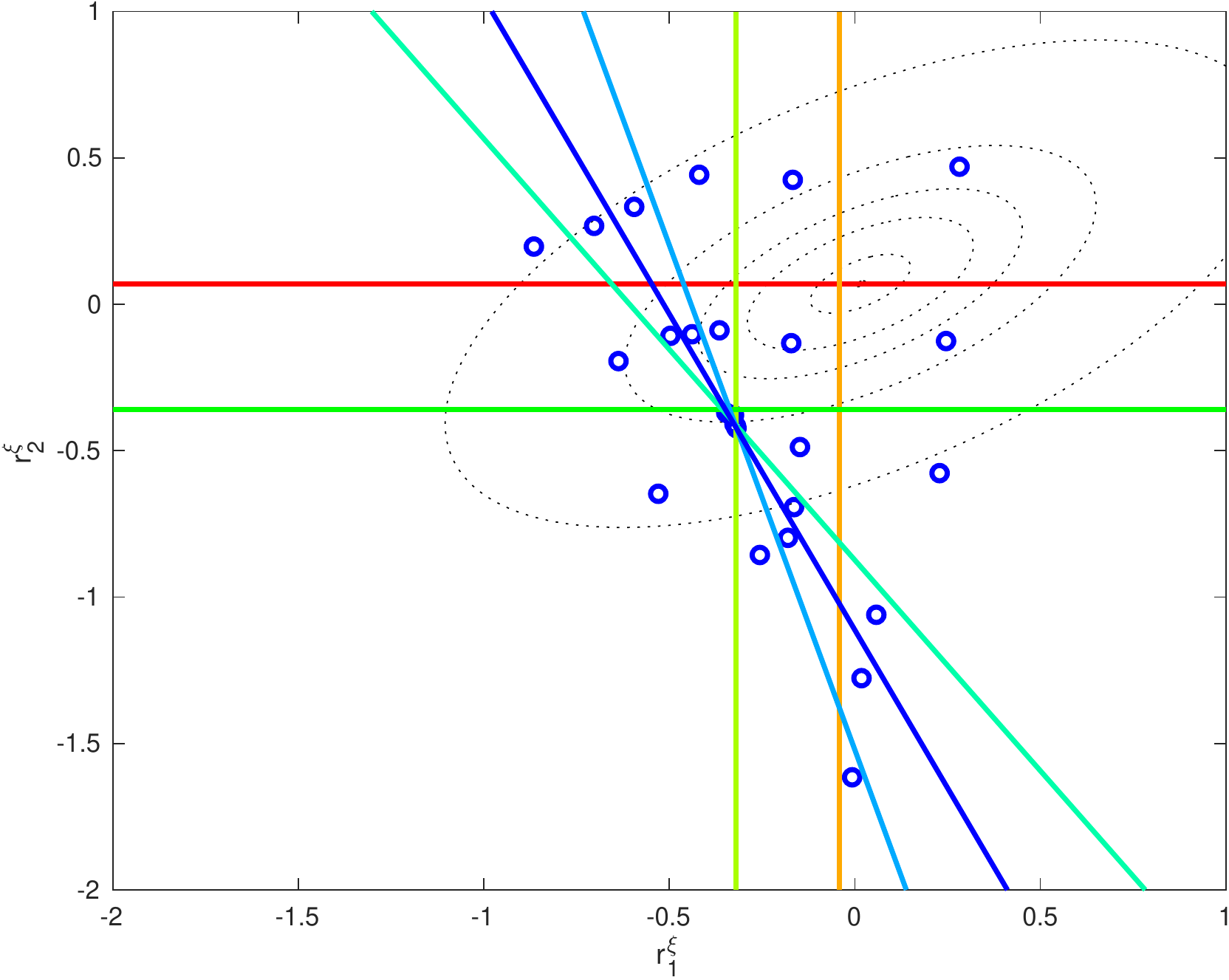}
		\caption{Iteration 7}
	\end{subfigure} 
	~%
	\begin{subfigure}[b]{0.31\textwidth}
		\includegraphics[width=\linewidth]{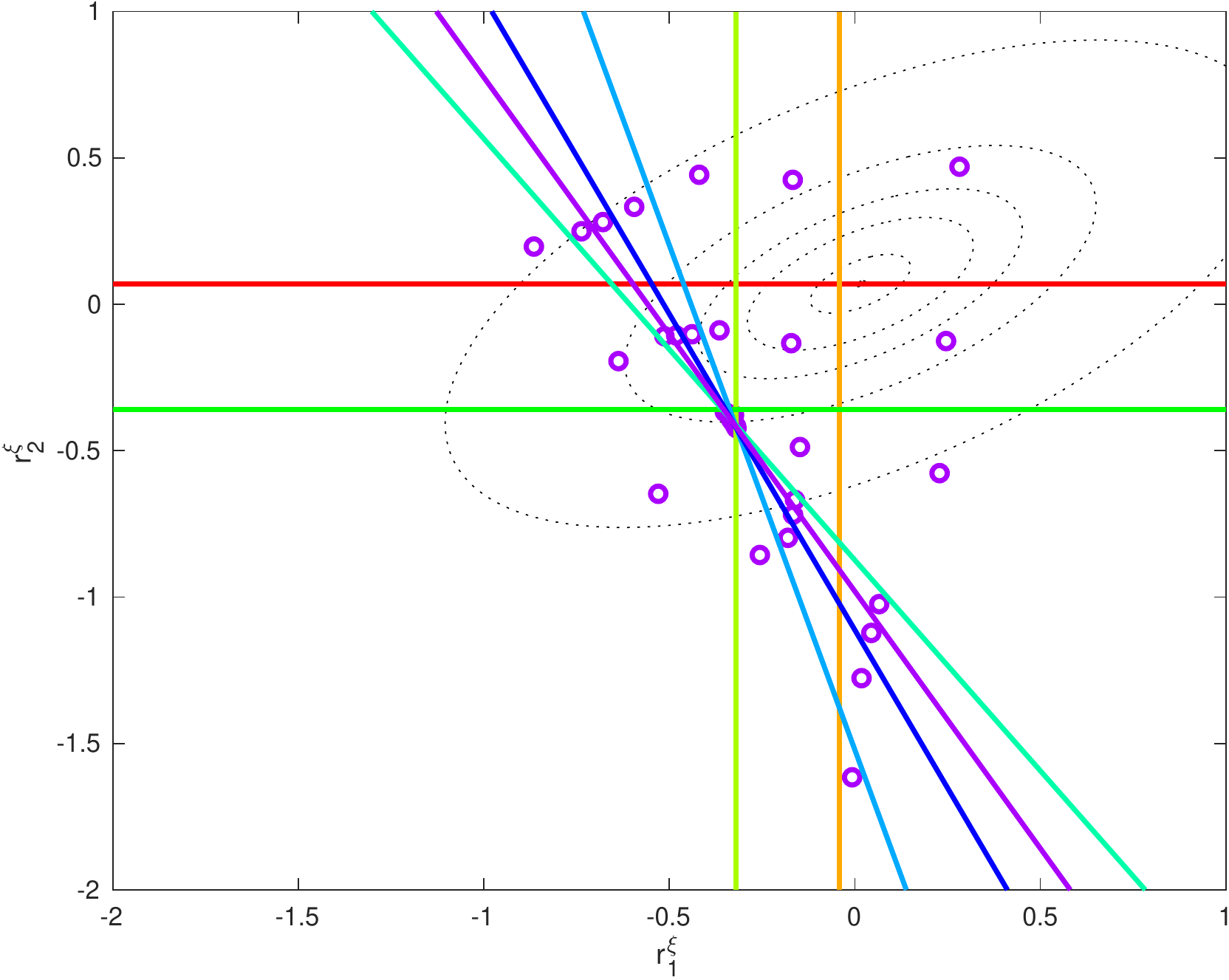}
		\caption{Iteration 8}
	\end{subfigure} 
	~%
	\begin{subfigure}[b]{0.31\textwidth}
		\includegraphics[width=\linewidth]{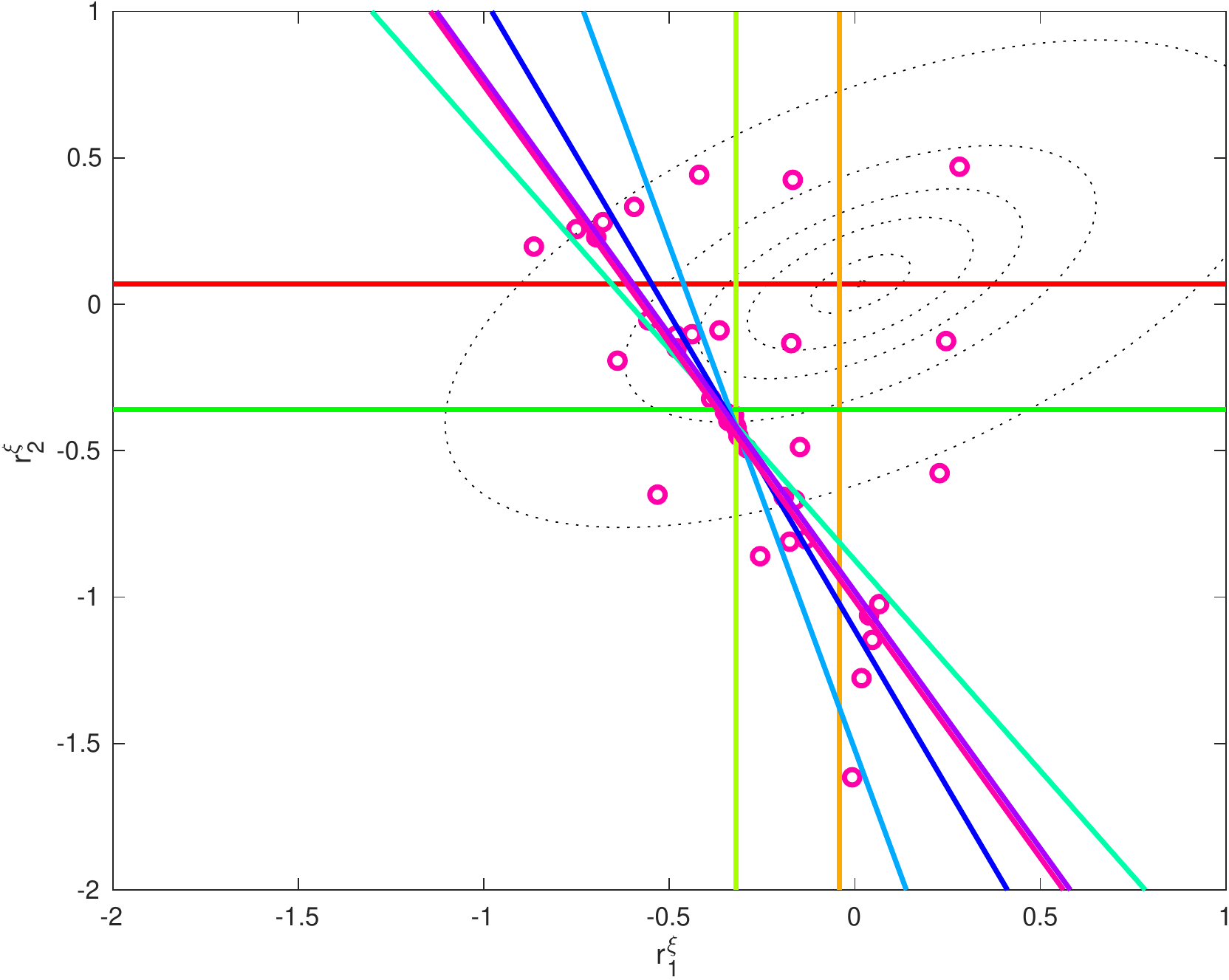}
		\caption{Iteration 9}
	\end{subfigure} 
\caption{Partition of $\Omega$ for the CVaR portfolio example}\label{fig:cvar2s0.1a}
\end{figure}

We solve the problem using two stocks and a risk level of $\delta=0.1$ to provide a graphical representation of the algorithm. To estimate the probabilities and expected return of each region, we use a  Monte Carlo sampler of the underlying distribution. \cref{tab:cvar2s0.1a} shows the results for our instance. We can notice that the problem converges quickly to the optimal solution, as well as in the previous LandS example. A more detailed analysis can be seen in \cref{fig:cvar2s0.1a}, where the region $\Omega$ is presented, with ellipses indicating the 50\%, 80\%, 90\%, 95\% and 99\% confidence intervals of the normal bidimensional distribution. In the first 4 iterations, the aggregated model considers only the riskiest scenario (bottom left dot) and invests the entire portfolio in the stock with the highest return $r'$. Our algorithm generates a cut that divides the uncertainty region into $r_i^\xi \geq r'$ and  $r_i^\xi \leq r'$, where $i$ is the stock where the budget is invested. After Iteration 5, the portfolio starts to combine stocks, and the region of interest is divided more precisely to obtain a better estimation of the optimal problem solution.

At last, we remark that, in both computational examples, several algorithmic improvements can be implemented to solve larger and more complex problems (e.g., reaggregating regions with the same duals, considering only the last $k$ cuts, or subdividing only the active regions; see~\cite{munoz2018study} for more details). Nevertheless, our purpose is simply to show how the method can automatically divide a continuous random space, to iteratively define the regions of interest for the problem and converge to the optimal solution.

\section{Conclusions}\label{sec:conclusions}
We present a generalization of the adaptive partition-based method for solving two-stage stochastic problems that contributes to extend the method to a more general setting, particularly, to consider continuous distributions of the uncertain parameters. The resulting algorithm allows to tackle this type of problems, by automatically disaggregating the uncertainty space and solving a discrete (tractable) problem in each iteration.  Naive computational experiments show the efficacy of the method to refine the uncertainty set in different regions of interest.  It is important to remark that the refining procedure depends considerably on the structure of the problem, but it is sufficiently general for a broad family of problems, namely generating a hyperplane which cuts and splits one or more regions in the current partition.  We strongly believe that this research represents a starting point for further development of computational methods for stochastic problems with continuous distributions. In particular, problems with high dimensional uncertainty and different continuous distribution could be challenging to compute conditional expectations and element probabilities, we suggest the reader to see related literature on numerical methods for this purpose in \cite{l2018generalized} and \cite{matthias2012simulating}.

\bibliographystyle{amsplain}
\bibliography{BibliografiaFinal}

\end{document}